\documentclass[10.5pt,reqno]{amsart}
\usepackage{longtable} 
\usepackage{hyperref}
\usepackage[T1]{fontenc}
\usepackage[utf8]{inputenc}
\usepackage[english]{babel} 
\usepackage{textcomp}
\usepackage{dsfont}
\usepackage{latexsym}
\usepackage{amssymb}
\usepackage{amsthm}
\usepackage{amsmath} 
\DeclareMathAlphabet{\mathpzc}{OT1}{pzc}{m}{en}
\usepackage{yfonts}
\usepackage{xfrac}
\usepackage{newlfont}
\usepackage{graphicx}
\usepackage{mathtools}
\usepackage{comment}
\usepackage{indentfirst}
\usepackage{braket}
\usepackage{mathrsfs} 
\usepackage{xcolor} 
\usepackage{fixmath}

\usepackage{etoolbox}

\newcommand{\mbeps}{\mathbold{\varepsilon}}

\usepackage{scalerel}[2014/03/10]
\usepackage[usestackEOL]{stackengine}
\newcommand{\dashint}{\,\ThisStyle{\ensurestackMath{%
			\stackinset{c}{.2\LMpt}{c}{.5\LMpt}{\SavedStyle-}{\SavedStyle\phantom{\int}}}%
		\setbox0=\hbox{$\SavedStyle\int\,$}\kern-\wd0}\int}

\DeclareMathOperator{\supp}{Supp}

\DeclareMathOperator{\Hol}{Hol}

\newcommand{\ee}{\mathrm{e}}

\newcommand{\vect}[1]{\mathbf{{#1}}}
\newcommand{\dd}{\mathrm{d}}

\DeclarePairedDelimiter{\abs}{\lvert}{\rvert}

\DeclarePairedDelimiter{\norm}{\lVert}{\rVert}

\let\originalleft\left
\let\originalright\right
\renewcommand{\left}{\mathopen{}\mathclose\bgroup\originalleft}
\renewcommand{\right}{\aftergroup\egroup\originalright}

\newcommand{\N}{\mathds{N}}

\newcommand{\C}{\mathds{C}}

\newcommand{\R}{\mathds{R}}

\newcommand{\Fc}{\mathcal{F}}

\newcommand{\Hc}{\mathcal{H}}

\newcommand{\Lc}{\mathcal{L}}
\renewcommand{\Mc}{\mathcal{M}}

\newcommand{\Sc}{\mathcal{S}}

\newcommand{\meg}{\leqslant}
\newcommand{\Meg}{\geqslant}
\newcommand{\eps}{\varepsilon}
\renewcommand{\phi}{\varphi}

\thanks{{\em Math Subject Classification 2020}:  42B30,  42B25, 32A35.}
\thanks{The author is a member of the 	Gruppo Nazionale per l'Analisi
	Matematica, la Probabilit\`a e le	loro Applicazioni (GNAMPA) of
	the Istituto Nazionale di Alta Matematica (INdAM)}

\begin{document}
	
	\title{A Note on Hardy Spaces on Polyhedral Tube Domains}
	
	\author[M.\ Calzi]{Mattia Calzi} 
	\address{Dipartimento di Matematica, Universit\`a degli Studi di
		Milano, Via C. Saldini 50, 20133 Milano, Italy}
	\email{{\tt mattia.calzi@unimi.it}}
	
	\theoremstyle{definition}
	\newtheorem{deff}{Definition}[section]

	\newtheorem{oss}[deff]{Remark}
	
	\newtheorem{ass}[deff]{Assumptions}
	
	\newtheorem{nott}[deff]{Notation}

	\theoremstyle{plain}
	\newtheorem{teo}[deff]{Theorem}
	
	\newtheorem{lem}[deff]{Lemma}
	
	\newtheorem{prop}[deff]{Proposition}
	
	\newtheorem{cor}[deff]{Corollary}
	
	\begin{abstract}
		We provide a description of the real variable Hardy spaces $H^p$, $p\in (0,1]$ associated to a regular polyhedral open convex cone, devised as real variable analogues of the holomorphic Hardy spaces on the corresponding tube domain.
	\end{abstract}
	
	\maketitle
	
	\section{Introduction}
	
	In~\cite{Grafakosetal}, given a regular polyhedral open convex cone $\Omega \subseteq \R^n$ and the associated tube domain $T\coloneqq \R^n+i \Omega$, the holomorphic Hardy space
	\[
	H^1(T)=\Set{f\in \Hol(T)\colon \sup_{y\in \Omega} \norm{f(\,\cdot\,+i y)}_{L^1(\R^n)}<+\infty}
	\]
	was considered, and several equivalent characterizations, via the unrestricted non-tangential maximal function, the area function, and the $g$-function, were considered. As observed in~\cite[Remark 1.2]{Grafakosetal}, several proofs only relied on the corresponding `real-variable' Hardy spaces on $\R^n$, but the authors were unable to complete the proof of the equivalence of the above-mentioned conditions for these `real-variable' Hardy spaces. We propose to address this gap, adding an additional equivalent condition and considering all real Hardy spaces with exponent $p\in (0,1]$.
	
	The missing step is based on the characterization of (vector-valued) real Hardy spaces on $\R$ via the Hilbert transform, and allows us to provide an additional description of the real Hardy space by means of the Fourier transform, clarifying their structure. The equivalence of the above-mentioned conditions is stated in Propositions\ref{prop:1b},~\ref{prop:2b},~\ref{prop:4b}, and~\ref{prop:5b}.
	
\section{Preliminaries}
	
\subsection{Sobolev Spaces}
	
	Here we introduce some notation for the Sobolev spaces of infinite order and provide some remarks that will be useful in the sequel.
	
	\begin{deff}
		Take $p\in [1,\infty]$ and a Hilbert space $H$. We define $W^{\infty,p}(\R^n;H)$ as the space of $f\in L^p(\R^n;H)$ such that $\partial^\alpha f \in L^p(\R^n;H)$ for every $\alpha\in \N^n$, endowed with the corresponding topology. We define $W^{-\infty,p}(\R^n;H)$ as the dual of the closure of $C^\infty_c(\R^n;H)$ in $W^{\infty,p'}(\R^n;H)$ (that is, the dual of $W^{\infty,p'}(\R^n;H)$ unless $p=1$).
	\end{deff}
	
	\begin{oss}
		Observe that, by the Sobolev embeddings, $W^{\infty,p}(\R^n;H)\subseteq W^{\infty,q}(\R^n;H)$ continuously for $p\meg q$. Arguing by transposition, we then see that $W^{-\infty,p}(\R^n;H)\subseteq W^{-\infty,q}(\R^n;H)$ for $p\meg q$.
	\end{oss}
	
	\begin{oss}
	By means of Young's inequality, it is readily seen that convolution induces continuous bilinear mappings
	\[
	W^{-\infty,p_1}(\R^n;H)\times W^{-\infty,p_2}(\R^n)\to W^{-\infty,p_3}(\R^n;H)
	\]
	and
	\[
	W^{-\infty,p_1}(\R^n;H)\times W^{\infty,p_2}(\R^n)\to W^{\infty,p_3}(\R^n;H)
	\]
	for every $p_1,p_2,p_3\in [1,\infty]$ such that $\frac{1}{p_1'}+\frac{1}{p_2'}=\frac{1}{p_3'}$.
	\end{oss}
	
	\begin{oss}
	Notice that, by~\cite[Theorem XXV of Chapter VI, \S\ 8]{Schwartz}, $W^{-\infty,\infty}(\R^n)$ is the space of bounded distributions and $\bigcup_{q\in [1,\infty)} W^{-\infty,q}(\R^n)$ is the space of distributions which are restricted at infinity, in the terminology of~\cite{Stein}.  
	\end{oss}
	
\subsection{The Classical Case}

In this section, we recall some basic features of the theory of classical real Hardy spaces on $\R$.

Define the non-tangential approach region $\Gamma_\alpha(x)\coloneqq \Set{(y,t)\colon t>0, \abs{y-x}\meg \alpha t}$ for every $\alpha\Meg 0$ and  $x\in \R$.
Define the Poisson kernel $(p_t)_{t\Meg 0}$ so that $p_0=\delta_0$ and $p_t(x)= \frac{t}{\pi(x^2+t^2)}$ for $t>0$ and $x\in \R$. 

Observe that $p_t\in W^{\infty,1}(\R)$ for every $t>0$, so that we may define $f*p_t\in W^{\infty,q}(\R^n;H)$ for every $f\in W^{-\infty,q}(\R^n;H)$, for every (real) Hilbert space $H$, and for every $q\in [1,\infty]$. 

\begin{oss}
	If $f\in W^{-\infty,p}(\R;H)$ for some $p\in [1,\infty)$, then $f*p_t\to 0$ for $t\to +\infty$ in $W^{\infty,\infty}(\R;H)$. 
	
	In fact, we have a continuous bilinear mapping $W^{-\infty,p}(\R;H)\times W^{\infty,p'}(\R)\to W^{\infty,\infty}(\R;H)$ and $p_t\to 0$ for $t\to +\infty$ in $W^{\infty,p'}(\R)$.
\end{oss}

Given $f\in W^{-\infty,\infty}(\R^n;H)$, define the non-tangential (radial, for $\alpha=0$) maximal-function
\[
(\Mc_\alpha f)(x)= \sup_{(y,t)\in\Gamma_\alpha(x) } \abs{(f*p_t)(y)},
\]
the area function
\[
(\Sc_\alpha f)(x)= \bigg(\int_{\Gamma_\alpha(x)} \abs{\nabla(f*p_t)(y)}^2 \,\dd \Hc^2(y,t)\bigg)^{1/2}
\]
for $\alpha>0$ (where, by abuse of notation, $\nabla$ denotes the gradient in the variables $(y,t)$), and the $g$-function
\[
(\Sc_0 f)(x)= \bigg(\int_{\Gamma_0(x)}\abs{\nabla(f*p_t)(x)}^2 \,\dd \Hc^1(t)\bigg)^{1/2}.
\]
Here, $\Hc^1$ and $\Hc^2$ denote the one- and two-dimensional Hausdorff measures (that is, Lebesgue measure on $\R$ and $\R^2$, respectively).

Define $h_\pm \coloneqq \Fc^{-1}\chi_{\pm(0,+\infty)}$, where $\Fc$ denotes the Fourier transform. Observe that $h_++h_-=\delta_0$ and that $h_+-h_-$ is (up to a multiplicative constant) the convolution kernel of the Hilbert transform. In the following result we shall use convolution with $h_\pm$ instead of the Hilbert transform since we wish to emphasize the relation with holomorphic functions. Clearly, $f*h_\pm$ is a distribution with values in the complexification  $H_\C$ of~$H$.

\begin{teo}\label{teo:1b}
	Take $p\in (0,\infty)$, $\alpha,\alpha' \Meg 0$, and a Hilbert space $H$.
	Given $f\in \bigcup_{q\in [1,\infty)} W^{-\infty,q}(\R;H)$, the following conditions are equivalent:
	\begin{enumerate}
		\item[\textnormal{(1)}] $\Mc_\alpha f\in L^p(\R)$;
		
		\item[\textnormal{(2)}] $\Sc_{\alpha'} f \in L^p(\R)$;
		
		\item[\textnormal{(3)}] the $f*h_+*p_t, f*h_-*p_t$, $t\in (0,+\infty)$, are uniformly bounded in $L^p(\R;H_\C)$.
	\end{enumerate}
	More precisely, there is a constant $C>0$ such that
	\[
	\begin{split}
		\norm{\Mc_\alpha f}_{L^p(\R)}&\meg C \norm{\Sc_{\alpha'} f}_{L^p(\R)}\\
		&\meg C^2\sup_{t>0}\Big( \norm{f*h_+*\phi_t}_{L^p(\R;H_\C)}+  \norm{f*h_-*p_t}_{L^p(\R;H_\C)}\Big)\\
		&\meg C^3\norm{\Mc_\alpha f}_{L^p(\R)}
	\end{split}
	\]
	for every $f\in \bigcup_{q\in [1,\infty)} W^{-\infty,q}(\R;H)$. 
\end{teo}

Notice that, since $f\in  \bigcup_{q\in [1,\infty)} W^{-\infty,q}(\R;H)$ and $h_-,h_+\in\bigcap_{q\in (1,\infty]} W^{-1,q}(\R)$, both   $f*h_-$, and $f*h_+$ are well defined elements of $\bigcap_{q\in (1,\infty)} W^{-\infty,q}(\R;H)$. 

\begin{proof}
	The equivalence of (1) and (2) follows from~\cite[Corollaries 1--3 to Theorem 9]{FeffermanStein} since, by the previous remarks, $U\colon (x,t)\mapsto (f*p_t)(x)$ is a harmonic function on $\R\times \R_+^*$ which tends to $0$ for $t\to +\infty$ (uniformly in $x$). 
	The equivalence of (1) and (3) then follows from~\cite[Theorem 9]{FeffermanStein}, since $v\colon (x,t)\mapsto [f*(h_+-h_-)*p_t](x)$ is the harmonic conjugate of $u$.
	Notice that, even though the above references are stated for harmonic functions with values in $\R$, it is more or less explicitly observed therein that they hold for functions with values in a general Hilbert space (cf., e.g.,~\cite[Remark 1 of \S\ 7]{FeffermanStein}).
\end{proof}

\section{Results}

Let $e_1,\dots, e_m$ be unit vectors in $\R^n$, with $m\Meg n$, such that $(e_j)_{j\in J}$ is a basis of $\R^n$ for every $J\subseteq\Set{1,\dots, m}$ with $n$ elements. Observe that, in particular,  $\bigcap_{j=1}^m e_j^\perp=\Set{0}$. For every $\mbeps\in \Set{-1,1}^m$, define 
\[
U_{\mbeps}\coloneqq \Set{\lambda \in \R^n\colon \forall j=1,\dots,m \quad \eps_j \langle \lambda, e_j\rangle \Meg 0},
\]
and observe that each $U_{\mbeps}$ is a closed convex cone and that $\R^n=\bigcup_{\mbeps} U_{\mbeps}$. Let $A$ be the set of $\mbeps\in \Set{-1,1}^m$ such that $U_\mbeps$ has a non-empty interior, and observe that $\R^n=\bigcup_{\mbeps\in A} U_\mbeps$, up to a $\Hc^n$-negligible set.

\begin{oss}
	Notice that $A$ may be strictly smaller than $\Set{-1,1}^m$. For example, if we consider the vectors $(0,1), (1,0), 2^{-1/2}(1,1)$, then $A=\Set{\pm(1,1,1), \pm(1,-1,1),\pm(1,-1,-1)  }$.
\end{oss}

For every $\mbeps\in A$, define  $\Omega_\mbeps$ as the interior of the polar $U_\mbeps^\circ$ of $U_\mbeps$, and observe that $\Omega_\mbeps$ is a non-empty regular open convex cone: in fact, $\Omega_\mbeps$ is regular since $U_\mbeps$ has a non-empty interior; in addition, $\Omega_\mbeps$ is non-empty since $U_\mbeps$ does not contain straight lines.

Now, take $\mbeps\in \Set{-1,1}^m$, and define
\[
\pi_\mbeps\colon \R^m\ni (x_1,\dots, x_m)\mapsto \sum_j \eps_j x_j e_j\in \R^n
\]
and, for every $\vect t\in \R_+^m$,
\[
p_{ \vect t}\coloneqq (\pi_{\vect 1_m})_*(p_{t_1}\otimes \cdots \otimes p_{t_m})
\]
(push-forward of integrable measures, where each $p_t$ with $t>0$ is identified with a measure by means of Lebesgue measure). Notice that, since $p_t$ is symmetric for every $t\Meg 0$, we have
\[
p_{\vect t}=(\pi_{\mbeps})_*(p_{t_1}\otimes \cdots \otimes p_{t_m})
\]
for every $\vect t\in \R_+^m$ and for every $\mbeps\in \Set{-1,1}^m$.

\begin{oss}
	$\Omega_\mbeps=\pi_\mbeps((\R_+^*)^m)$ for every $\mbeps\in A$. In fact, by definition, $U_\mbeps$ is the polar of $\bigcup_{j=1,\dots,m} \R_+ \eps_j e_j$ so that, by the bipolar theorem, $U_\mbeps^\circ$ is the closed convex cone generated by $e_1,\dots, e_m$. 
\end{oss}
 
Now, for every $\alpha \Meg 0$ and for every $x\in \R^n$ define the unrestricted non-tangential approach region
\[
\Gamma_\alpha(x)\coloneqq \Set{ \bigg(x+\sum_{j=1}^m \lambda_j e_j,\vect t\bigg)\colon \vect t \in (\R_+^*)^m, \forall j=1,\dots, m\quad \abs{\lambda_j}\meg \alpha t_j}
\]
and, for every $f\in  W^{-\infty,\infty}(\R^n)$, the unrestricted non-tangential maximal function
\[
(\Mc_\alpha f)(x)=\sup_{(y,\vect t)\in \Gamma_\alpha(x)} \abs{(f*p_{\vect t})(y)}.
\] 

Next, define
\[
H_\mbeps\coloneqq \Fc^{-1}(\chi_{U_\mbeps})
\]
for every $\mbeps\in A$. Observe that
\[
\sum_{\eps\in A} \Fc H_\mbeps=1
\]
almost everywhere by the previous remarks, so that $\sum_{\eps\in A} H_\mbeps=\delta_0$.

\begin{oss}\label{oss:1b}
	For every $\eps\in A$, one has $H_\mbeps\in \bigcap_{q\in (1,\infty)}W^{-\infty,q}(\R^n)$.
\end{oss}

Consequently, we may consider $f*H_\mbeps$ for every $f\in \bigcup_{q\in [1,\infty)}W^{-\infty,q}(\R^n;H)$ and for every $\mbeps\in A$; the convolution then still belongs to $\bigcup_{q\in [1,\infty)}W^{\infty,q}(\R^n;H)$. Observe that, for such $f$, we have $f= \sum_{\mbeps\in A} f*H_\mbeps$, a fact which would look quite mysterious from the Fourier transform side, since $\Fc f$ need \emph{not} be a function.

\begin{proof}
	Define $h_{\eps_j,e_j}\in \Sc'(\R^n)$, for every $j=1,\dots, m$, so that
	\[
	\Fc h_{\eps_j,e_j}= \chi_{\eps_j(0,+\infty)}(\langle \,\cdot\,, e_j\rangle ).
	\]
	Then, $\Fc H_\mbeps=\Fc h_{\eps_1,e_1}\cdots \Fc h_{\eps_m,e_m}$. It will therefore suffice to show that $h_{\eps_j,e_j}\in W^{-\infty,q}(\R^n)$ for every $q\in(1,\infty)$ and for every $j=1,\dots, m$. By~\cite[Theorem XXV of Chapter VI, \S\ 8]{Schwartz}, it will suffice to show that $\phi* h_{\eps_j,e_j}\in L^q(\R^n)$ for every $\phi\in C^\infty_c(\R^n)$. Up to a change of coordinates, we may assume that $e_j=(1,0,\dots, 0)$. Then, $h_{\eps_j,e_j}=h_{\eps_j}\otimes \delta_0$, so that the assertion follows from the fact that $h_{\eps_j}\in W^{-1,q}(\R)$.
	
	Notice that, since we now know that $h_{\eps_j,e_j}\in \bigcap_{q\in (1,\infty)}W^{-\infty,q}(\R^n)$ for every $j=1,\dots, m$, we also deduce that $H_\mbeps=h_{\eps_1,e_1}*\cdots* h_{\eps_m,e_m}$.
\end{proof}

\begin{deff}
	Take $\mbeps\in A$ and $p\in (0,\infty]$ and define $T_\mbeps\coloneqq \R^n+i \Omega_\mbeps$ and $H^p(T_\mbeps)$ as the space of holomorphic functions $f\colon T_\mbeps\to \C$ such that
	\[
	\norm{f}_{H^p(T_\mbeps)}\coloneqq \sup_{y\in \Omega_\mbeps} \norm{f(\,\cdot\,+i y)}_{L^p(\R^n)}
	\]
	is finite, endowed with the corresponding quasi-norm.
\end{deff}

The following result generalizes~\cite[Theorems 1.1, 3.1 and 3.2]{Grafakosetal}. The proof is essentially the same.

\begin{prop}\label{prop:1b}
	Take $p\in (0,\infty)$.  Take $f\in  W^{-\infty,\infty}(\R^n)$ so that  $\Fc f$ is supported in $U_\mbeps$. Then, there is a unique holomorphic function $g\colon T_\mbeps\to \C$ such that   
	\[
	g(x+i \pi_\mbeps(\vect t))= (f*p_{\vect t})(x)
	\]
	for every $x\in \R^n$ and for every $\vect t\in (\R_+^*)^m$. In particular,
	\[
	\norm{g}_{H^p(T_\mbeps)}=\sup_{\vect t \in (\R_+^*)^m} \norm{f*p_{\vect t}}_{L^p(\R^n)}.
	\]
\end{prop}

\begin{proof}
	Assume first that $f\in W^{-\infty,2}(\R^n)$, and define $g$ so that
	\[
	g(x+i y)\coloneqq \frac{1}{(2\pi)^n}\int_{U_{\mbeps}} (\Fc f)(\xi) \ee^{i\langle \xi, x+i y\rangle}\,\dd \Hc^n(\xi)
	\]
	for every $x+ i y\in T_\mbeps$; notice that the integral makes sense since $(1+\abs{\,\cdot\,}^{2N})\Fc f\in L^2(\R^n)$ for some $N\in\N$.  It is readily seen that $g$ is holomorphic. Since, in addition,  $(\Fc p_{\vect t})(\xi)=\ee^{i\langle \xi, x+i \pi_\mbeps(\vect t)\rangle} $ for every $\xi\in U_\mbeps$, it is clear that $g(x+i \pi_\mbeps(\vect t))= (f*p_{\vect t})(x)$	for every $x\in \R^n$ and for every $\vect t\in (\R_+^*)^m$, whence the conclusion in this case.
	
	Now, assume that $f\in W^{-\infty,\infty}(\R^n)$. Take $\phi\in C^\infty_c(\R^n)$ so that $\supp \phi\subseteq U_\mbeps$, define $\psi_k\coloneqq (\Fc^{-1} \phi)(2^{-k}\,\cdot\,)$, and observe that $f \psi_k\in W^{-\infty,2}(\R^n)$ and that $\Fc(f\psi_k)=\Fc f* \Fc \psi_k$ is still supported in the cone $U_\mbeps$ for every $k\in\N$. We may then define  holomorphic functions $g_k$  as above so that
	\[
	g_k(x+ i \pi_\mbeps(\vect t))=((f \psi_k)*p_{\vect t})(x)=\langle f, \psi_k p_{\vect t}(x-\,\cdot\,)\rangle
	\]
	for every $x\in \R^n$, for every $\vect t \in (\R_+^*)^m$, and for every $k\in\N$. Since clearly $\psi_k p_{\vect t}(x-\,\cdot\,)$ converges to $ p_{\vect t}(x-\,\cdot\,)$ in $W^{1,1}(\R^n)$, uniformly as $(x,\vect t)$ runs through a compact subset of $\R^n\times (\R_+^*)^m$, it is readily seen that $g_k$ converges locally uniformly to a holomorphic function $g$. The assertion follows.
\end{proof}

Now, define the unrestricted non-tangential  maximal function
\[
(\Mc_\alpha f)(x)\coloneqq \sup_{(y,\vect t)\in\Gamma_\alpha(x)} \abs{(f*p_{\vect t})(y)}
\]
the area function
\[
(\Sc_\alpha f)(x)\coloneqq \bigg(\int_{\Gamma_\alpha(x)} \abs{\nabla_1\cdots \nabla_m (f*p_{\vect t})(y)}^2\frac{t_1\cdots t_m}{\Hc^n(R(x,\vect t))}\,\dd \Hc^{n+m}(y,\vect t)  \bigg)^{1/2}
\]
for $\alpha>0$, and the $g$-function
\[
(\Sc_0 f)(x)\coloneqq \bigg(\int_{\Gamma_0(x)} \abs{\nabla_1\cdots \nabla_m (f*p_{\vect t})(x)}^2 t_1\cdots t_m\,\dd \Hc^{m}(\vect t)  \bigg)^{1/2}.
\]
Here, by some abuse of notation, $\nabla_j=(\partial_{e_j},\partial_{t_j})$ for every $j=1,\dots,m$, while
\[
R(x,\vect t)=\Set{x+\sum_{j=1}^m \lambda_j e_j\colon \forall j=1,\dots, m\quad \abs{\lambda_j}\meg t_j}
\]
for $x\in \R^n$ and $\vect t\in (\R_+^*)^m$.

\begin{prop}\label{prop:2b}
	Take $p\in (0,\infty)$ and $\alpha\Meg 0$. Then, there is a constant $C>0$ such that 
	\[
	\norm{\Mc_\alpha f}_{L^p(\R^n)}\meg C \sum_{\mbeps\in A}\sup_{\vect t\in (\R_+^*)^m}\norm{ f*H_\mbeps*p_{\vect t}}_{L^p(\R^n)}
	\] 
	for every $f\in \bigcup_{q\in [1,\infty)}W^{-\infty,q}(\R^n)$.
\end{prop}

\begin{proof}
	Observe that, by the previous remarks,  $f=\sum_{\mbeps\in A} f*H_\mbeps$. Therefore, it will suffice to prove the assertion assuming that $f=f*H_\mbeps$ and $\sup_{\vect t\in (\R_+^*)^m}\norm{ f*p_{\vect t}}_{L^p(\R^n)}<+\infty$ for some $\mbeps\in A$. Then, Proposition~\ref{prop:1b} shows that there is a unique  $g\in H^p(T_\mbeps)$ such that
	\[
	g(x+i \pi_\mbeps(\vect t))=(f*p_{\vect t})(x)
	\]
	for every $x\in \R^n$ and for every $\vect t\in (\R_+^*)^m$. Take $q\in (0,\min(1,p))$. By the plurisubharmonicity of $\abs{g}^q$, we then see that
	\[
	\begin{split}
	\abs{(f*p_{\vect t+\vect t'})(x)}^q&=\abs{g(x+i \pi_\mbeps(\vect t+\vect t'))}^q\\
		&\meg(\abs{g(\,\cdot\,+i \pi_\mbeps(\vect t+(0, t'_2,\dots, t'_m)))}^q*p_{(t'_1,0,\dots,0)})(x)\\
		&\meg \dots\\
		&\meg (\abs{g(\,\cdot\,+i \pi_\mbeps(\vect t))}^q*p_{\vect t'})(x)\\
		&=(\abs{f*p_{\vect t}}^q*p_{\vect t'})(x)
	\end{split}
	\]
	for every $x\in \R^n$ and for every $\vect t,\vect t'\in (\R_+^*)^m$. Consequently, there is a constant $C_1>0$ such that
	\[
	(\Mc_\alpha f)^q \meg C_1\Mc_{\mathrm{it}} \abs{f}^q,
	\]
	where $\Mc_{\mathrm{it}}$ denotes the iterated Hardy--Littlewood maximal function defined in~\cite[Section 3.1]{Grafakosetal}.
	The assertion then follows using the boundedness of the iterated maximal function.
\end{proof}

\begin{prop}\label{prop:3b}
	Take $p\in (0,\infty)$ and $\alpha,\alpha'> 0$. Then, there is a constant $C>0$ such that 
	\[
	\norm{\Mc_\alpha f}_{L^p(\R^n)}\meg C\norm{\Mc_{\alpha'} f}_{L^p(\R^n)}
	\] 
	for every $f\in W^{-\infty,\infty}(\R^n)$.
\end{prop}

\begin{proof}
	The proof is similar to that of~\cite[Lemma 3.1]{Grafakosetal}.
\end{proof}

\begin{prop}\label{prop:4b}
	Take $p\in (0,1]$ and $\alpha,\alpha'> 0$. Then, there is a constant $C>0$ such that 
	\[
	\norm{\Sc_\alpha f}_{L^p(\R^n)}\meg C\norm{\Mc_{\alpha'} f}_{L^p(\R^n)}
	\] 
	for every $f\in \bigcup_{q\in (1,\infty)} W^{-\infty,q}(\R^n)$. In addition, $\Sc_0 f \meg C\Sc_\alpha f$.
\end{prop}

\begin{proof}
	\textsc{Step I} Assume first that $f\in L^1(\R^n)$, that $\Mc_{\alpha'} f\in L^1(\R^n)\cap L^p(\R^n)$, and that $p<2$. Then,~\cite[Theorem 4.1]{Grafakosetal} shows that there is a constant $C_1>0$ such that, if $\alpha'$ is sufficiently large,\footnote{Actually, the cited reference explicitly considers only the case $\alpha=1$, but the reader may verify that the same arguments work for general $\alpha>0$. }
	\[
	\Hc^n(\Set{x\in \R^n\colon (\Sc_\alpha f)(x)>\lambda})\meg C_1 \Hc^n(\Set{x\in \R^n\colon (\Mc_{\alpha'} f)(x)>\lambda})+\frac{C_1}{\lambda^2} \int_{\Set{x\in \R^n\colon (\Mc_{\alpha'} f)(x)\meg \lambda}} (\Mc_{\alpha'} f)^2\,\dd \Hc^n
	\]
	for every $\lambda>0$ and for every $f$ as above. Multiplying by $p\lambda^{p-1}$ and integrating in $\lambda$ we then infer that
	\[
	\begin{split}
	\norm{\Sc_\alpha f}_{L^p(\R^n)}^p&\meg C_1 \norm{\Mc_{\alpha'} f}_{L^p(\R^n)}^p+pC_1 \int_0^{+\infty} \lambda^{p-3} \int_{\Set{x\in \R^n\colon (\Mc_{\alpha'} f)(x)\meg \lambda}} (\Mc_{\alpha'} f)^2\,\dd \Hc^n\,\dd \lambda \\
		&=C_1 \norm{\Mc_{\alpha'} f}_{L^p(\R^n)}^p+pC_1 \int_{\R^n} (\Mc_{\alpha'} f)(x)^2\int_{(\Mc_{\alpha'} f)(x)}^{+\infty} \lambda^{p-3}  \,\dd \lambda\,\dd \Hc^n(x)\\
		&=C_1 \norm{\Mc_{\alpha'} f}_{L^p(\R^n)}^p+\frac{pC_1}{2-p} \int_{\R^n} (\Mc_{\alpha'} f)(x)^p   \,\dd \Hc^n(x),
	\end{split}
	\]
	whence the conclusion in this case.
	
	\textsc{Step II} Now, assume that $p\meg 1$. Take $\vect t\in (\R_+^*)^m$ and observe that $f*p_{\vect t}\in W^{\infty,\infty}(\R^n)$, so that clearly $\Mc_{\alpha'}(f*p_{\vect t})\in L^\infty(\R^n)$. Since clearly $\abs{f*p_{\vect t}}\meg\Mc_{\alpha'} (f*p_{\vect t})\meg \Mc_{\alpha'} f$, we have $f*p_{\vect t},\Mc_{\alpha'} (f*p_{\vect t})\in L^p(\R^n)\cap L^\infty(\R^n)$. Then, by~\textsc{step I},
	\[
	\norm{\Sc_\alpha(f*p_{\vect t})}_{L^p(\R^n)}\meg \Big(\frac{2C_1}{2-p}\Big)^{1/p} \norm{\Mc_{\alpha'}(f*p_{\vect t})}_{L^p(\R^n)}\meg \Big(\frac{2C_1}{2-p}\Big)^{1/p} \norm{\Mc_{\alpha'}f}_{L^p(\R^n)}.
	\]
	By Fatou's lemma, passing to the limit for $\vect t\to 0$ we then get to 
	\[
	\norm{\Sc_\alpha f}_{L^p(\R^n)}\meg  \Big(\frac{2C_1}{2-p}\Big)^{1/p} \norm{\Mc_{\alpha'}f}_{L^p(\R^n)}.
	\]
	The last assertion was proved in~\cite[Proposition 6.1]{Grafakosetal} when $f\in L^1(\R^n)$. The extension to general $f$ may be obtained as above.
\end{proof}

We shall now complete the picture with the following result.

\begin{prop}\label{prop:5b}
	Take $p\in (0,1]$. Then, there is a constant $C>0$ such that 
	\[
	\sum_{\mbeps\in A}\sup_{\vect t\in (\R_+^*)^m}\norm{f*H_\mbeps*p_{\vect t}}_{L^p(\R^n)}\meg C \norm{\Sc_0 f}_{L^p(\R^n)}
	\]
	for every $f\in \bigcup_{q\in [1,\infty)}W^{-\infty,q}(\R^n)$.
\end{prop}

\begin{proof}
	\textsc{Step I}
	Fix $\mbeps\in A$, and define $h_{\eps_j,e_j}\in \Sc'(\R^n)$, for every $j=1,\dots, m$, so that 
	\[
	\Fc h_{\eps_j,e_j}= \chi_{\eps_j(0,+\infty)}(\langle \,\cdot\,, e_j\rangle ).
	\]
	By the proof of Remark~\ref{oss:1b}, we see that $h_{\eps_j,e_j}\in \bigcap_{q\in (1,\infty)}W^{-\infty,q}(\R^n)$ for every $j=1,\dots, m$, and that $H_\mbeps=h_{\eps_1,e_1}*\cdots* h_{\eps_m,e_m}$. In addition, it is readily seen that $h_{\eps_j,e_j}$ is a convolutor of $L^q(\R^n)$ for every $j=1,\dots, n$ and for every $q\in (1,\infty)$.
	 
	Assume  that $f\in L^q(\R^n)$ for some $q\in (1,\infty)$. 
	Observe that, if $h$ is any convolutor of $L^q(\R^n)$, then $f*h\in L^q(\R^n)$ and $f*h*p_{\vect t}\to f*h$ in $L^q(\R^n)$ and almost everywhere for $\vect t\to 0$, $\vect t \in(\R_+^*)^m$ (cf.~\cite[Theorem 3.1]{Grafakosetal}). In particular, 
	\[
	\sup_{\vect t\in (\R_+^*)^m}\norm{f*h*p_{\vect t}}_{L^p(\R^n)}=\norm{f*h}_{L^p(\R^n)}.
	\]  
	
	Define $X_1\coloneqq L^2((\R_+^*)^{m-1}, t_2\cdots t_m \dd (t_2,\dots, t_m))$ and observe that, since $f\in L^q(\R^n)$ and $\Sc_0 f\in L^p(\R^n)$, for $\Hc^{n-1}$-almost every $x'\in e_1^\perp$ the function $g_{x'}\colon \R\ni x \mapsto [(t_2,\dots,t_m)\mapsto (f*p_{(0,t_2,\dots, t_m)})(x'+ x e_1)]\in X_1$ satisfies
	\[
	\norm*{\int_{(0,+\infty)} \abs{\nabla(g_{x'}*p_t)(x)}^2 t\,\dd t}_{L^p_x(\R;X_1)}=\norm{(\Sc_0 f)(x'+x e_1)}_{L^p_x(\R)}<\infty,
	\]
	(where we wrote $\norm{G(x)}_{L^p_x(\R^n)}$ instead of $\norm{G}_{L^p(\R^n)}$ to simplify the notation),
	that is, $g_{x'}$ belongs to $H^p(\R;X_1)$, thanks to Theorem~\ref{teo:1b}. More precisely, taking into account the previous remarks, we may find a constant $C_1>0$ (independent of $f$ and $x'$) such that
	\[
	\norm*{ g_{x'}*\Fc^{-1}\chi_{\eps_1(0,+\infty)} }_{L^p(\R;X_1)}\meg  C_1\norm{(\Sc_0 f)(x'+x e_1)}_{L^p_x(\R)}.
	\]
	Integrating in $x'$, we then infer that
	\[
	\norm*{g*h_{\eps_1,e_1}}_{L^p(\R^n;X_1)}\meg C_1\norm{\Sc_0 f}_{L^p(\R^n)},
	\] 
	where $g\colon \R^n\ni x\mapsto [(t_2,\dots, t_m)\mapsto (f*p_{(0,t_2,\dots, t_m)})(x)]\in X_1$.
	Iterating this procedure, we then infer that
	\[
	\norm{f*H_\mbeps}_{L^p(\R^n)}\meg  C_1^m\norm{\Sc_0 f}_{L^p(\R^n)}.
	\]
	
	\textsc{Step II} Now, consider the general case. Applying the previous arguments to $f*p_{ \vect t}$, which clearly belongs to $L^q(\R^n)$ for some $q\in (1,\infty)$, we see that
	\[
	\sum_{\mbeps\in A}\norm{f*p_{\vect t}*H_\mbeps}_{L^p(\R^n)}\meg (2 C_1)^m\norm{\Sc_0(f*p_{\vect t})}_{L^p(\R^n)}\meg (2C_1)^m \norm{\Sc_0 f}_{L^p(\R^n)}
	\]
	for every $\vect t\in (\R_+^*)^m$, whence the conclusion.
\end{proof}

\begin{oss}
	When $p\in (1,\infty)$, it is readily seen that $H_\mbeps$ is a convolutor of $L^p(\R^n)$ for every $\mbeps\in A$, so that $f\in L^p(\R^n)$ if and only if ($f\in \bigcup_{q\in (1,\infty)} W^{-\infty,q}(\R^n)$ and) $f*H_\mbeps\in L^p(\R^n)$ for every $\mbeps\in A$. By Proposition~\ref{prop:2b}, it then follows that $\Mc_\alpha f\in L^p(\R^n)$ if and only if $f\in L^p(\R^n)$,  for every $\alpha \Meg 0$. The continuity of $\Sc_\alpha$ on $L^p(\R^n)$ may be obtained arguing as in the proof Propositions~\ref{prop:3b}, with minor adjustments, when $p\in (1,2)$; the assertion then follows by duality when $p\in (2,\infty)$, and by interpolation (or directly) when $p=2$. In order to conclude, one only has to show that $\Sc_0$ is invertible in $\Lc(L^p(\R^n))$ (so that the same holds for $\Sc_\alpha$ for every $\alpha>0$). This may be obtained arguing as in the proof of Proposition~\ref{prop:5b}.
\end{oss}

Let us now draw a few consequences of the previous results.
First, it is now completely clear that the real  Hardy spaces $H^p_{e_1,\dots, e_m}(\R^n)=\Set{f\in \bigcup_{q\in (1,\infty)}W^{-\infty,q}(\R^n)\colon \Mc_1 f\in L^p(\R^n)}$, $p\in (0,1]$, actually depend only on the convex cones $U_\mbeps$, $\mbeps\in A$: any other family of unit vectors in $\R^n$ inducing the same system of convex cones will define the same Hardy spaces, and conversely. This leads to a more precise statement and a substantially simpler proof of~\cite[Proposition 3.7]{Grafakosetal}. In addition, the proof of~\cite[Theorem 1.2]{Grafakosetal} may be extended to cover the case $p\in (0,1]$, instead of only $p=1$: it suffices to apply, with the usual arguments, the equivalence of the various characterizations of the real Hardy spaces $H^p_{e_1,\dots, e_m}(\R^n)$. Moreover, the long and delicate arguments of~\cite[Section 5 and the Appendix]{Grafakosetal} are no longer necessary to prove the implication (4) $\implies $ (1) in~\cite[Theorem 1.2]{Grafakosetal}.

\end{document}